\documentclass[a4paper,12pt]{amsart}
\usepackage{color}
\usepackage{amsmath}
\usepackage{amssymb}
\usepackage{amsthm}
\usepackage[mathscr]{eucal}
\usepackage{xcolor}

\numberwithin{equation}{section}

\newcommand{\diam}{\operatorname{diam}}

\newcommand{\R}{\operatorname{\mathbb{R}}}

\newcommand{\mLaplace}{\boldsymbol{\Delta}}

\theoremstyle{plain}
\newtheorem{thm}{Theorem}[section]
\newtheorem{lem}[thm]{Lemma}
\newtheorem{prop}[thm]{Proposition}
\newtheorem{cor}[thm]{Corollary}

\theoremstyle{definition}

\theoremstyle{remark}

\begin{document}

\title{Ricci  curvature in dimension 2}

\author{Alexander Lytchak}
\address{Alexander Lytchak:
Mathematisches Institut,
Universit\"at K\"oln,
Weyertal 86--90, 50931, K\"oln, Germany}
\email{alytchak@math.uni-koeln.de}

\author{Stephan Stadler}
\address{Stephan Stadler: Mathematisches Institut der Universit\"at M\"unchen, Theresienstr. 39, D-80333 M\"unchen, Germany}
\email{stadler@math.lmu.de}




\begin{abstract}
We prove that in two dimensions the synthetic notions of lower bounds on sectional and on Ricci curvature coincide.
\end{abstract}

\maketitle

\section{Introduction}
\subsection{The result} In this note we provide an affirmative answer to a question raised by C\'{e}dric Villani, \cite[Open Problem 9]{Villani}.

\begin{thm} \label{thm: main}
Let $(X,d)$ be a metric space and let
$\mathcal H^2$ be the $2$-dimensional Hausdorff measure on $X$. If $(X,d,\mathcal H^2)$ is an
$RCD (\kappa, 2)$ space
then $X$ is an Alexandrov space  of curvature at least $\kappa$.	
\end{thm}

The converse of our theorem is due to  Anton Petrunin,  \cite{Petrunin}.

 As a consequence of Theorem \ref{thm: main},  all $RCD(\kappa ,2)$  spaces $(X,d,\mathcal H^2)$ are topological surfaces, possibly with boundary.

\subsection{Theory of RCD spaces}
The synthetic theory of spaces with lower Ricci curvature bounds, initiated in \cite{Sturm} and \cite{Lott}, has experienced a tremendous growth over the last 15 years.
We refer the reader to the surveys \cite{Villani} and \cite{Ambrosio-review} for an overview of the current state of the theory and the huge bibliography.

The analytic aspects of the theory are  highly developed. Most results, previously known on Riemannian manifolds with
lower bounds on Ricci curvature, now possess natural generalizations to the synthetic framework, see for instance \cite{AGS},
\cite{Gigli-memoirs}, \cite{Mcav}.  Moreover, there is a sophisticated understanding of measure theoretic, or rather ``almost everywhere'' properties  of such spaces, \cite{Mondino}, \cite{Constancy}.

On the other hand, gaining information on the ``everywhere'' metric structure, or even less ambitious, on the topological structure of RCD spaces,  seems  to be very difficult. Only a few results in this direction are known, all of which require either a rigid setting, \cite{Gigli}, \cite{Volumerig}, or strong a priori assumptions on the geometry, \cite{KK}, \cite{lowdimRCD}.


\subsection{Central steps in the proof}
The proof is a combination of analytic RCD-techniques, some basic results about Alexandrov spaces, the  canonical uniformization  of singular surfaces introduced in \cite{LW-parameter} and Reshetnyak's analytic theory of surfaces with integral curvature bounds, \cite{Reshetnyak-GeomIV}.

In the first step, we use the Bishop--Gromov comparison and  the cone-rigidity theorem, \cite{Volumerig}, to study the infinitesimal structure of a space $X$ as in Theorem \ref{thm: main}. More precisely, it is shown that unique tangent spaces exist at all points and are Euclidean cones over either intervals or circles.   In analogy with the theory of Alexandrov spaces, we call the set of points, at which the tangent space is homeomorphic to a half-plane, the \emph{boundary} of $X$ and denote it by $\partial  X$.  Another application of the Bishop--Gromov comparison implies that $\partial X$ is a closed subset of $X$.
This is the content of Section 2.

In Section 3, we consider, for any  $\delta >0$, the set  of points $X^{\delta}$ at which
the density of $X$ is $\delta$-close to the density of the Euclidean plane.
The set $X^{\delta}$ is open, dense, disjoint from $\partial X$ and, by
the theorem of Cheeger--Colding--Reifenberg, \cite{CC},  $X^{\delta}$ is a topological surface without boundary, once $\delta$ is small enough.   Due to \cite{Mondino-isop},  $X^{\delta}$ satisfies an isoperimetric inequality
of Euclidean type.

The next two sections are the central pieces of the proof.
In Section 4, we use \cite{LW-parameter}, to find for  any point $x\in X^{\delta}$  a small closed
neighborhood $\bar \Omega$ and a canonical  parametrization $\phi:\bar D\to \bar \Omega$ obtained  by solving  a Plateau problem.
This parametrization is a \emph{quasisymmetric map}, which in addition is \emph{infinitesimally conformal}. This implies, in particular, that there exists a function  $f\in L^2(D)$, the \emph{conformal factor}, such that
the lengths of \emph{almost all curves} in $\bar \Omega$ are obtained by integrating $f$ along the corresponding curve in $\bar D$.  In other words, the metric in $\bar \Omega$ appears to be obtained from the flat metric in
$\bar D$ by a ``singular conformal change with factor $f$''.

In Section 5, we identify the Laplace operator of $\Omega$ in terms of $D$ and $f$ and use the  Bakry--\'Emery property of
RCD spaces to derive an analytic condition on $f$. This condition implies
 log-subharmonicity in the case $\kappa =0$ and a related property for general $\kappa$.

  Using \cite{Reshetnyak-GeomIV},
	we  deduce in Section 6 that the metric  on $\Omega$ is {\em everywhere} controlled by
$f$, and has curvature bounded below by  $ \kappa$ in the sense of Alexandrov.

In order to finish the proof of the main theorem  we
 show that $X^{\delta}$ is a convex subset of $X$,   a statement which is expected to be true  in much greater generality, \cite{CN}.  Once the convexity is verified, Toponogov's globalization theorem, in the form of \cite{Petrunincomplete},  provides that last step.   The proof of this convexity in Section 7, relies
 on the non-branching property of RCD spaces, \cite{Rajala-Sturm}.  Assuming that a geodesic between two non-boundary points passes through  a point $x$ on the boundary $\partial X$, we easily obtain many branching geodesics, once we know that $X$ is a topological surface with boundary near $x$.  While in general, topological control  in RCD spaces  is very difficult to achieve, here we obtain the required statement by a twofold application of the Cheeger-Colding--Reifenberg theorem.


\subsection{Acknowledgment} The paper arose as a follow-up project to a series of joint papers of the first author and  Stefan Wenger. For this and
inspiring discussions the   authors express their gratitude to Stefan Wenger.
We are grateful to Vitali Kapovitch and Anton Petrunin for helpful conversations.
 For helpful comment on  previous versions we would like to thank Nicola Gigli, Martin Kell, Christian Ketterer and Andrea Mondino.  
Both authors were partially supported by DFG grant  SPP 2026.

\section{Basic structure} \label{sec: 2}
\subsection{Notation}
By $D$  we denote the open  unit disc in $\R^2$ and by $\bar D$ its closure. By $\mathcal H^k$ or $\mathcal H^k_X$ we denote the $k$-dimensional Hausdorff measure on a metric space $X$.

In a metric space $X$ we denote by $d$ the distance and by
  $B_r(x)$ the  open ball of radius $r$ around the point $x$.
By $\ell (\gamma) =\ell _X(\gamma)$ we will denote the length of a curve $\gamma$
in a metric space $X$.


\subsection{Setting} We assume some familiarity with the synthetic theory of lower Ricci curvature bounds.  In particular, we assume that the reader is familiar with the notion of
$RCD(\kappa, N)$ spaces, which we will not define.


  In the rest of the paper
we fix a  space $X$ satisfying the assumptions of Theorem \ref{thm: main}.
Thus, $X$ is a geodesic, locally compact metric space. By assumption, the space
$X$ is $RCD(\kappa, 2)$ with respect to the reference measure $\mathcal H^2$,
whose support is $X$.
In the terminology introduced in \cite{ncRCD}, this means that $X$ is  a \emph{non-collapsed}
$RCD(\kappa ,2)$ space.

For $x\in X$ and $r>0$ we set
\begin{equation} \label{eq: bg}
b(x,r) =\mathcal H^2 (B_r(x)) \,.
\end{equation}
Recall that  the balls satisfy the \emph{Bishop--Gromov property}, thus,
for all $s<r$, the quotient $b(x,r)/b(x,s)$ is bounded
from above by the corresponding quotient in the $2$-dimensional simply connected Riemannian manifold $M^2_{\kappa}$ of constant curvature $\kappa$, \cite{Sturm}.

Therefore, for any  $x\in X$ we have a well-defined positive \emph{density}
$$0< b(x):= \lim _{r\to 0} \frac { b(x,r)} {\pi r^2}  \leq 1\, ,$$
 see \cite[Corollary 2.13]{ncRCD}.

Again by the Bishop-Gromov property, the space $X$ is \emph{locally $2$-Ahlfors regular}. Thus,
for every  compact subset $K$ of $X$ there exists some $L>0$, such that
 all $x\in K$ and all $0<r<1$ we
have
$$\frac 1 L \cdot r^2\leq \mathcal H^2 (B_r(x)) \leq L\cdot r^2 \, .$$

\subsection{Convergence and blow-ups}
Any sequence  $(Y_i,y_i)$  of non-collapsed $RCD(\kappa, 2)$-spaces
has a subsequence converging in the pointed Gromov--Hausdorff topology
to a limit space $(Y,y)$. This limit space $Y$ is an $RCD(\kappa ,2)$-space
with respect to a limiting measure $\mu$, see \cite{AGS}, \cite{GMS}.  If there is $\epsilon >0$ such that $\mathcal H^2 (B_1 (y_i)) \geq \epsilon$, for all $i$,  then the limit space $Y$ is non-collapsed and $\mu=\mathcal H^2$,
\cite[Theorem 1.2]{ncRCD}.

 Moreover, in this situation the densities behave semi-continuously:
$$b(y) \leq \liminf b(y_i) \,.$$
 In particular, the density function $b:X\to (0,1]$ is lower semi-continuous.

As a special case of \cite[Theorem 1.2]{ncRCD}, for any sequence $x_i\in X$ converging to a point $x\in X$ and any
sequence $r_i$ of positive numbers converging to $0$, there is a subsequence,
such that  the rescaled spaces $(\frac 1 {r_i} X, x_i)$ converge to a
non-collapsed $RCD(2,0)$ space $(Y,y)$.   We call any such space $(Y,y)$ a
\emph{blow-up} of the space $X$.  By the continuity  of $\mathcal H^2$
with respect to the convergence
and by  the Bishop--Gromov inequalities in $X$, we have, for all $z\in Y$ and all $s\geq 0$,
\begin{equation} \label{eq: limitvol}
\mathcal H^2 (B_s(z)) \geq \pi \cdot b(x) \cdot s^2 \, .
\end{equation}

\subsection{Tangent cones}
A blow-up for a constant sequence of base points $x_i=x$  is called a tangent cone
of $X$ at the point $x$.  Relying on the volume rigidity of metric cones, \cite{Volumerig}, it has been shown in \cite[Proposition 2.7]{ncRCD},  that
any such tangent cone $T$ of $X$ at the point $x$ is isometric to a Euclidean cone
$C(Z)$ over a compact  space  $Z$. Moreover, $Z$ is an $RCD(0,1)$ space of diameter at most $\pi$, \cite{Ketterer}.

As has been proved in \cite{lowdimRCD}, such a space $Z$ is a closed interval or a circle.   Due to the stability
of $\mathcal H^2$ and the definition of  $b(x)$, the density
$b(x)$ coincides with the density of $CZ$ at the vertex $0_x$ of the cone $CZ$.
Therefore,
$$b(x)=\frac 1 {2\pi} \mathcal H^1 (Z) \, .$$

Thus, at any point $x\in X$ at most two different tangent cones can exist.

However, the following lemma, valid for all doubling metric measure spaces and  probably well-known to experts shows that
there cannot exist exactly two different tangent cones.

\begin{lem}
 The space $\mathcal T ^x $ of all metric measure tangent cones $(Y,y,\mathcal H^2)$  at the point $x \in X$ is a connected subset 
of the space $\mathbb X$  of isomorphism classes of all pointed, proper metric measure  spaces, \cite{GMS}.
\end{lem}

\begin{proof}
Consider the map $F:(0,1]\to \mathbb X$ which sends a number $r>0$ onto  the pointed metric measure space $(\frac 1 r X,x, \mathcal H^2)$.  The map  is continuous and the image  of $F$ 
has  compact closure  in $\mathbb X$.
 
The set $\mathcal T^x$ consists of all limit points $\lim _{r_i\to 0} F(r_i)$  of this continuous curve. Thus, $\mathcal T^x$ is an intersection
of  a nested sequence of compact connected sets, the closures of the tails of $F$.  Hence, $\mathcal T^x$ is connected.
\end{proof}

Therefore, we have shown:
\begin{cor}
At any  $x\in X$ there exists a unique tangent cone.   This tangent
cone is a Euclidean cone $CZ$, where $Z$ is  a circle or an interval.
Moreover, $\mathcal H^1 (Z)=  b(x)\cdot 2\pi$ and  the diameter of $Z$ is at most $\pi$.
\end{cor}

We denote the unique tangent cone at the point $x\in X$ by $T_x=T_xX$.
By above $T_x$ is homeomorphic to a plane or to a half-plane.


\subsection{The boundary}
We define \emph{the boundary} $\partial X$ of $X$ to be the set of all points
$x\in X$ for which  the tangent space $T_xX$ is homeomorphic to a half-plane.

\begin{lem} \label{lem: boundclosed}
The boundary $\partial X$ is a closed subset of $X$ which contains only points $x$ with $b(x)\leq \frac 1 2$.
\end{lem}

\begin{proof}
	By definition, $x\in \partial X$ if and only if $T_x$ is isometric to $CZ$, where $Z$ is an interval. Since $Z$ has length $b(x)\cdot 2\pi$ and   diameter  at most $\pi$, this implies  $b(x)\leq \frac 1 2$.
	

Let  $x\in X\setminus \partial X$ be arbitrary. Then $T_x$ is a cone over a circle, thus for any point $z$ in $T_x$ but the vertex $0_x$, the density of
$T_x$ at $z$ is $1$.  For any sequence $x_i\in X\setminus \{x \}$ converging to $x$,
we choose $r_i=d(x,x_i)$. Then (possibly after choosing a subsequence) under the convergence of $(\frac 1 {r_i} X,x)$ to $(T_x,0_x)$ the points $x_i$ converge to a point $z$ at distance $1$ from $0_x$.  The semi-continuity of densities implies that
$\lim_{i\to \infty} b(x_i)=1$.  In particular, $x_i$ is not in $\partial X$, for $i$ large enough.

Hence,  $X\setminus \partial X$ must be open.
\end{proof}

As a  consequences of the splitting theorem we obtain:

\begin{lem} \label{lem: nogeod}
Let $x\in X$ be a point which is an interior point of a geodesic.
If $x\in \partial X$ then $T_x$ is isometric to the Euclidean half-plane.
If $x\in X\setminus \partial X$ then $T_xX$ is isometric to $\R^2$.
\end{lem}

\begin{proof}
The assumption implies  that $T_xX$ contains a line (the tangent space to the geodesic through $x$).  By the splitting theorem  \cite{Gigli},  the space $T_x$
splits off a line. This implies the claim, since $T_x$ is a cone over an interval or a circle.
\end{proof}

\section{Almost regular parts}
 For any $\delta >0$, we call a point $x\in X$
a \emph{$\delta$-regular point} if $b(x)>1-\delta$.  By $X^{\delta}$ we denote the set of all $\delta$-regular points in $X$.  We have the following discreteness statement:

\begin{lem} \label{lem: delta}
The set $X^{\delta}$ is open in $X$, for any $\delta>0$.	For any $\delta < \frac 1 2 $, the set $X^{\delta}$ is disjoint from  $\partial X$ and the complement $(X\setminus \partial X) \setminus X^{\delta}$ is discrete in $X\setminus \partial X$.
\end{lem}

\begin{proof}
	The semi-continuity of the density function shows that $X^{\delta}$ is open.
	Due to Lemma \ref{lem: boundclosed}, the set  $X^{\delta}$ is disjoint from
	$\partial X$, for $\delta <\frac 1 2$.
	
	Finally, the last argument in the proof of  Lemma \ref{lem: boundclosed} implies that any point in $X\setminus \partial X$ has a punctured neighborhood completely contained in $X^{\delta}$. This implies the last claim.
\end{proof}


The following observation is a very special and rough case of the results
obtained in \cite{CC} and \cite{Mondino-isop}.


\begin{lem}  \label{lem: mondino}
	There exists some $\delta, C  >0$     with the following property.
	Every point $x\in X^{\delta }$ has a neighborhood  $U_x$ homeomorphic
	to $D$, such that for any subset $K$ of $U_x$ homeomorphic to $\bar D$,
	 \begin{equation} \label{eq: mondino}
	\mathcal H^2 (K) \leq   C\cdot ( \mathcal H^1 (\partial K)) ^2 \,.
	\end{equation}
\end{lem}

  \begin{proof}
  	Due to \cite[Theorem 1.2]{ncRCD} and the Cheeger--Colding--Reifenberg theorem, \cite[Theorem A.1]{CC},  there exists some
  	$\delta >0$, such that any point $x\in X^{\delta}$ has a neighborhood
  	 $U_x$  homeomorphic to  $D$.

  It remains to prove, that, for sufficiently small $\delta$, we have the isoperimetric inequality \eqref{eq: mondino}, for any
   closed topological disc $K$ in some neighborhood of any point in $X^{\delta}$.

  	This  statement is proved   in \cite[Corollary 1.6]{Mondino-isop}
  	(in a much more general and precise form) with $ \mathcal H^1 (\partial K)^2$
  	on the right hand side of \eqref{eq: mondino} replaced by   $m^2(K)$, where   $m(K)$ is
  	the \emph{outer Minkowski content}
  	$$m(K):=\liminf _{r\to 0} \frac {\mathcal H^2(B_r(K))- \mathcal H^2(K)} {r}\, .$$
  	Here $B_r(K)$ denotes the set of  points  with distance at most $r$
  	to $K$.
  	
   It suffices to prove $4\pi \cdot   \mathcal H^1 (\partial K) \geq m(K)$
  	for any subset $K$ of $X$ homeomorphic to a closed disc, compare \cite[Theorem 3.2.39]{Federer}.
		
		If $\partial K$ is not rectifiable there is nothing to prove.
  	Otherwise, for any natural $n$,  we  set $r=\frac 1 {2n} \cdot \mathcal H ^1 (\partial K)$ and find an $r$-dense subset  $A_n$ in   $\partial K$  with $n$ points.
  	Then $B_r(\partial K) \subset B_{2r } (A_n)$.
  	
  	By the Bishop--Gromov property, we deduce, for $r$ small enough,
  	$$\mathcal H^2 (B_r (\partial K )) \leq \mathcal H^2 (B_{2r } (A_n)) \leq n\cdot  2\cdot \pi \cdot (2r)^2 =4\pi \cdot r \cdot \mathcal H^1 (\partial K) \, .$$
  	This implies $m(K)  \leq 4\pi \cdot \mathcal H^{1} (\partial K)$ and finishes the proof.
  \end{proof}

 \section{Conformal parametrization} \label{sec: 4}
\subsection{Choice of a domain}  Let $\delta, C>0$ be as in Lemma \ref{lem: mondino}, let $x_0\in X^{\delta}$ be arbitrary and let $U_{x_0}$ be an  open neighborhood of $x_0$ in $X$ provided by  Lemma \ref{lem: mondino}.

Any Jordan curve $\Gamma$ in $U_{x_0}$ determines a \emph{Jordan domain}
$\Omega \subset U_{x_0}$ homeomorphic to $D$, such that $\bar \Omega =\Omega \cup \Gamma$ is homeomorphic to $\bar D$.

Starting with any  Jordan curve  $\Gamma$ in $U_{x_0}$ whose Jordan domain
$\Omega$ contains $x_0$, we can replace $\Gamma$ by a nearby curve and assume that $\Gamma$ is biLipschitz to the round circle $S^1$, see \cite[Lemma 4.3]{LW-parameter}.

We fix this curve $\Gamma$ and domain $\Omega$ for the rest of the section.

\subsection{Metric properties of $\Omega$}
Consider the set $\bar \Omega$ with its intrinsic metric $d_{\Omega}$.
Since $\Gamma$ is a biLipschitz embedding of a circle, the metric $d_{\Omega}$
is biLipschitz to the induced metric $d_X$ on $\bar \Omega$. Moreover,
$d_{\Omega}$ and $d_X$ locally coincide on $\Omega$.

In order to apply the parametrization results of \cite{LW-parameter}, we need to make sure that
$\bar \Omega$ is Ahlfors $2$-regular and is
 \emph{linearly locally connected} in the following sense.


A \emph{continuum} is a compact connected space.  A  metric  space $Y$ is called \emph{linearly locally connected} if there exists a positive constant $0<C<1$, such that for all $y\in Y$ and all $0<r<\diam (Y)$ the following holds true.  Any pair of points $z_1,z_2\in B_{C\cdot r} (y)$ is contained in a continuum $P \subset B_{r} (y)$; any  pair of points $z_1,z_2 \in Y\setminus B_r(y)$  is contained in a continuum $P\subset Y\setminus B_{C\cdot r}$.

If the space $Y$ is geodesic, the first condition is always satisfied. Linear local
connectedness is preserved under biLipschitz transformations.

\begin{lem} \label{lem: ahllin}
The space
 $(\bar \Omega, d_{\Omega})$ is $2$-Ahlfors regular and linearly locally connected.
\end{lem}

\begin{proof}
Since the $\mathcal H^2$-measures with respect to $d_X$ and $d_{\Omega}$ coincide on $\bar \Omega$, a quadratic upper bound on the $\mathcal H^2$-area of balls
in $\Omega$ follows from the corresponding upper bound on the area of
balls in $X$.

The existence of a lower quadratic bound on the area of balls in $\bar \Omega$
is essentially proved in \cite[Theorem 9.4]{LW-intrinsic}, as a consequence
of the quadratic isoperimetric inequality \eqref{eq: mondino}. We provide a simplified
version of the argument here.

Relying on the lower bound on the area of balls  in $X$,   it is
sufficient to find a constant $C_0 < 1$ with the following property: For all small $r$ and any $z\in \Gamma$ the ball
$B_{2r}(z)$ contains a point $y \in \Omega$ with distance at least $C_0 \cdot r$ to
$\Gamma$.

By topological reasons, the distance sphere of radius $r$ around $z$
must contain a continuum $P$, joining  two points on $\Gamma$, locally separated on $\Gamma$ by $z$.  From the biLipschitz property of the boundary curve
$\Gamma$, we now deduce the existence of $C_0<1$, such that the continuum $P$  contains a point $y$ as claimed above.

This finishes the proof of the Ahlfors $2$-regularity of $\bar \Omega$.

In order to prove that $\bar \Omega$ is linearly locally connected, recall first from  \cite[Theorem 8.6]{LW-intrinsic}, that the  isoperimetric property \eqref{eq: mondino} implies the following statement.
There exists  a constant $C_1>1$ such that,  for any $y\in \bar \Omega$ and any
$r>0$, the ball $B_r (y)$ is contractible inside  $B_{C_1\cdot r} (y)$.

The space
$\bar \Omega$ is geodesic and the boundary $\Gamma$ is linearly locally connected.
Thus it  suffices to prove the following claim. There exists some $C_2<1$, such that,  for every $y\in \bar \Omega$, every $r< \diam (\bar \Omega)$ and every  $z\in \bar \Omega \setminus B_r(y)$
there is a curve connecting $z$ with $\Gamma$ outside of $B_{C_2\cdot r} (y)$.

 Assume the contrary. By the Jordan curve theorem, there exists
 a Jordan curve $T$ in $B_{C_2 \cdot r} (y)$, whose Jordan domain contains the point $z$.  But this implies that $B_{C_2\cdot r} (y)$ is not contractible
in $B_r (y)$, which provides a contradiction to the result \cite[Theorem 8.6]{LW-intrinsic}, mentioned above, once $C_2$ is sufficiently small.

This finishes the proof of the Lemma.
\end{proof}

 \subsection{Moduli of families of curves and Newton-Sobolev maps} For the convenience of the reader,
we recall the notions of moduli of families of curves and of Newton--Sobolev maps with values in a metric space, \cite[Sections 5-7]{HKST}, in the  special case
used here.

Let $Y$ be a metric spaces with finite $\mathcal H^2 (Y)$.
%
For a family of curves $\mathcal{C}$ in $Y$, a Borel function $\sigma:Y\to [0,\infty]$ is called \emph{admissible for $\mathcal{C}$}
if $\int _{\gamma} \sigma  \geq 1$ for every locally rectifiable curve in $\mathcal{C}$. The \emph{modulus}  (more precisely the $2$-modulus) of the family   is  defined as
$$\mod (\mathcal{C}):=\inf_{\sigma} \int _Y \sigma ^2 \, d\mathcal H^2 \;,$$
where the infimum is taken over all Borel functions admissible for
$\mathcal{C}$.

 A statement holds for \emph{almost every curve}  in $Y$, if the family $\mathcal{C}  $ of all curves in $Y$  for which the statement does not hold has modulus $0$.

A measurable map $u:Y\to Z$ into a separable  metric space $Z$ is contained in the Newton--Sobolev space $N^{1,2} (Y,Z)$ if  for some
$z\in Z$ the composition $ d_z \circ u :Y\to \R$ is contained in $ L^2 (Y)$
and the following statement holds true.  There exists a function $\rho \in L^2 (Y)$,
called a \emph{weak upper gradient} of $u$, such that for almost any curve $\gamma$ in $Y$ the composition $u\circ \gamma$ is absolutely continuous and
\begin{equation} \label{eq: defns}
\ell (u\circ \gamma ) \leq  \int _{\gamma} \rho \;.
\end{equation}
There is a unique \emph{weak minimal upper gradient} $\rho _u \in L^2 (Y)$ of $u$
such that $\rho _u \leq \rho$ almost everywhere, for any \emph{weak upper gradient $\rho$} of $u$.
$$Ch (u) = Ch_Y ^2 (u):= \frac 1 2 \int _Y \rho _u ^2  \; d\mathcal H^2$$
is called the \emph{Cheeger energy}  of  $u$. In \cite{LW-parameter}, \cite{LW}
the equivalent notion of  \emph{Reshetnyak} energy $E_+ ^2 (u)= 2 \cdot Ch (u)$ has  been used.

\subsection{Canonical parametrization}
We will not recall the definition of \emph{quasisymmetric maps}, see \cite{HK}. Instead we will
use the theory of quasisymmetric maps as a black-box, providing references for each required
statement.

 From \cite[Theorem 1.1]{LW-parameter} and Lemma \ref{lem: ahllin} above we deduce:

\begin{cor}
Among all homeomorphisms $u:\bar D\to \bar \Omega$ there exist a homeomorphism
$\phi \in N^{1,2} (\bar D, \bar \Omega)$ with  minimal  Cheeger energy  $Ch _{\bar \Omega}  (\phi)$.
This homeomorphism $\phi$ is quasisymmetric.
\end{cor}

The inverse $\phi ^{-1} :\bar \Omega \to \bar D$  is quasisymmetric as well,
\cite[Proposition 10.6]{Hein}.  Since the disc $\bar D$ satisfies the \emph{$1$-Poincar\'e inequality}, \cite[Section 8]{HKST},  the space  $\bar \Omega$ satisfies the $q$-Poincar\'e inequality for some $q<2$, \cite[Theorem 2.3]{Koskela}.

 The map $\phi$ and  its inverse have \emph{Lusin's property $N$}, thus $\phi$ and $\phi ^{-1}$ preserve the class of sets of $\mathcal H^2$-measure $0$, \cite[Theorem 8.12]{Shan}.

 Due to \cite[Theorem 9.3]{HK},
$\phi\in N^{1,2} (\bar D , \bar \Omega)$ and
 $\phi^{-1}  \in N^{1,2} (\bar \Omega,  \bar D)$.
By \cite[Theorem 9.8]{Shan},
a family $\mathcal{C}$ of curves in $\bar D$ has modulus zero   if and only if the image $u\circ \mathcal{C}$  of this family
has modulus $0$ in $\bar \Omega$.

\subsection{Conformality and its consequences}
No tangent space of the space $X$ contains a non-Euclidean  normed vector space. Due to \cite[Proposition 11.2]{LW}, $X$ satisfies the property  ET, introduced in \cite[Definition 11.1]{LW}.  The metric  on $\Omega$ locally coincides with $d_X$. Since $\Gamma$ has vanishing $\mathcal H^2$-measure,
the space $\bar \Omega$ has the property ET as well.

Due to  \cite[Theorem 11.3]{LW}, the energy minimizer $\phi$   is a conformal map in the sense of \cite[Definition 6.1]{LW}, meaning that
almost all \emph{approximate metric differentials} of $\phi$ are multiplies of the (fixed) Euclidean norm on $\R^2$. Using \cite[Lemma 3.1]{LW-intrinsic} this reads as
 \begin{equation} \label{eq:length}
 \ell _{\Omega} (\phi \circ \gamma)= \int _{\gamma}  \rho _{\phi} \, ,
 \end{equation}
 for almost all curves $\gamma$ in $D$. Here, as before, $\rho _{\phi}$ denotes a weak minimal upper gradient of $\phi$.


 For any Borel subset $E\subset  \bar D$ we have, \cite[Lemma 3.3]{LW-intrinsic}
 \begin{equation} \label{eq: areachange}
 \mathcal H^2 _X (\phi (E))= \int _E \rho _{\phi} ^2   \, d\mathcal H^2 _{\bar D} \, .
 \end{equation}
 Since the inverse $\phi^{-1}$ has Lusin's property $N$, the weak minimal
 upper  gradient $\rho _{\phi}$ must be  positive almost everywhere.

 From \eqref{eq:length},  \cite[Lemma 3.1]{LW-intrinsic}  and the absolute continuity on
almost all curves of the
Sobolev maps $\phi, \phi ^{-1}$, \cite[Proposition 6.3.2]{HKST}, we deduce  that for any non-negative Borel function $h:\bar D\to \R$ and almost every curve $\gamma$ in $\bar D$
\begin{equation} \label{eq: h}
 \int _{\phi \circ \gamma } h \circ \phi ^{-1} = \int _{\gamma} \rho _{\phi} \cdot h \, .
 \end{equation}
Therefore, the Borel function
$$g:=\frac 1 {\rho _{\phi} \circ \phi ^{-1}} :\bar \Omega \to [0,\infty] $$
satisfies
$$\int _{\eta} g = \ell _{\bar D}(\phi ^{-1} \circ \eta )  \, ,$$
for almost every curve $\eta$ in $\bar \Omega$.  This implies that
$g$ is a weak minimal upper gradient of $\phi ^{-1}\in N^{1,2} (\bar \Omega, \bar D)$.





\section{Conformal factor } \label{sec: 5}
\subsection{Setting and aim}
We continue to use the notations from the previous section. Thus, we have a domain $\Omega \subset  X$ and a  conformal homeomorphism $\phi :\bar D\to \bar \Omega$, which is contained   in $N^{1,2} (\bar D,\bar \Omega)$.  We let $f:=\rho _{\phi}:\bar D\to [0,\infty]$ be a weak minimal upper gradient of $\phi$.

The aim of this section is the following result and its consequence, Corollary \ref{cor: conclusion}.

\begin{prop} \label{prop: main}
	The   function $f^{-2}$ is contained in $L^{\infty} _{loc} (D)$.
	For any harmonic function $v:D\to \R$ the distributional Laplacian
	$\Delta _D (f^{-2}  |\nabla v|^2)$ is a Radon measure on $D$
	and satisfies  as a  measure
	\begin{equation} \label{eq: proposition}
	\Delta _D (f^{-2} \cdot |\nabla v|^2) \geq  2\kappa \cdot |\nabla v|^2 \cdot \mathcal H^2_D\, .
	\end{equation}
\end{prop}

The proof will be a direct consequence of the Bakry-\'Emery inequality
on the $RCD(\kappa, 2)$ space $X$, once we have identified via the conformal homeomorphism $\phi$ the Sobolev spaces and Laplacians on $D$ with the corresponding objects on $X$.

\subsection{Identifications}
Whenever no confusion is possible we will use the homeomorphism $\phi ^{-1}$ to
identify $\bar \Omega $ with $\bar D$.

Due to  \eqref{eq: areachange}, under this identification we have
\begin{equation} \label{eq: areachange2}
\mathcal H^2_X |_{\Omega} = \mathcal H^2_\Omega  = f^2 \cdot \mathcal H^2 _D\, .
\end{equation}


We are going to identify the space of Sobolev functions $N^{1,2} (\Omega) =N^{1,2} (\Omega, \R)$ with the "classical" space $N^{1,2} (  D)= W^{1,2} (D)$.

From \eqref{eq: h}  we  draw the following conclusion.
Let $u: \Omega \to \R$  be measurable. A Borel function
 $\rho:\Omega \to [0,\infty]$ is a weak upper gradient of $u$, thus
 satisfies  \eqref{eq: defns} for almost every curve $\gamma $ in $\Omega$,
 if and only if $(\rho \circ \phi) \cdot f : D\to [0,\infty] $ is a weak upper gradient of
 the composition $u\circ \phi : D \to \R$.

 Due to \eqref{eq: areachange2}, $\rho\in L^2 ( \Omega)$ if and only if
 $ (\rho \circ \phi) \cdot f \in L^2 ( D)$.

  Since $\bar D$ and $\bar \Omega$ satisfy the $2$-Poincar\'e inequality, the
 $2$-integrability of a weak upper gradient implies that the function itself
 is $2$-integrable, \cite[Lemma 8.1.5, Theorem 9.1.2]{HKST}.

 This shows that a map $v:D \to \R$ is contained in
 $N^{1,2} (D)$ if and only if $u:=v\circ \phi^{-1}$ is contained in $N^{1,2} ( \Omega)$.
 Moreover, in this case  the weak minimal upper gradients satisfy
 \begin{equation} \label{eq: gradients}
 \rho _{u}  = \frac  {\rho_v} {f}\circ \phi ^{-1} =\frac { |\nabla v|} f\circ \phi^{-1}\, .
 \end{equation}
 Combining with \eqref{eq: areachange2} and identifying $D$ and $\Omega$ this shows
 \begin{equation} \label{eq: equalitym}
 |\nabla v|^2 \cdot \mathcal H^2 _D = \rho_u ^2\cdot \mathcal H^2_{\Omega} \; \; \text {and} \; \; Ch_{\Omega} (u) = Ch_{D} (v).
 \end{equation}


 \subsection{Laplacians}
By the RCD property, the spaces $X$ and $\R^2$ are \emph{infinitesimally Hilbertian},
meaning that $Ch_X$ and $Ch_{\R^2}$ are quadratic forms on $N^{1,2} (X)$ and
$N^{1,2} (\R^2)$, respectively. Thus, $D$ and $\Omega$ are infinitesimally Hilbertian as well, \cite[Proposition 4.22]{Gigli-memoirs}.

For $Y=X,\Omega, D, \R^2$ we consider the corresponding bilinear forms $\mathcal E_Y:N^{1,2} (Y)\times N^{1,2} (Y)\to \R$, called  \emph{Dirichlet forms},
$$\mathcal E _{Y}  (u_1,u_2) =\frac 1 2 (Ch_{Y} (u_1+u_2)^2 -Ch_{Y} (u_1-u_2)^2) \,.$$
From \eqref{eq: equalitym} we see that for $v_{1,2}\in N^{1,2} (D)$ and $u_{1,2}=v_{1,2}\circ \phi^{-1} \in N^{1,2} (\Omega)$
\begin{equation} \label{eq: dirichlet}
\mathcal E_{\Omega} (u_1,u_2 )= \mathcal E _D (v_1,v_2) \,.
\end{equation}
For $Y =D,\Omega$,  a function $u\in N^{1,2} (Y)$ is in the domain
of the (measure valued) Laplacian on $Y$  if there exists a Radon measure
$\nu$ on $Y$ with the following property, \cite[Definition 4.4, Proposition 4.7, Lemma 4.26]{Gigli-memoirs}: For all $w\in N^{1,2} (Y)$ which are
continuous and have compact support in $Y$ the equality
$$\mathcal E_Y (u,w)=- \int _Y w \; d\nu \, $$
holds true.  In this case, we set  $ \mLaplace _{Y} (u):= \nu$.

From \eqref{eq: dirichlet}, a function $u\in N^{1,2} (\Omega )$ is  contained in
the domain of the Laplacian, if and only if $v=u\circ \phi$ is in the domain of the Laplacian on  $D$. Moreover, in this case
\begin{equation} \label{eq: laplacian}
\mLaplace _D(v) =\mLaplace _{\Omega}(u) \, ,
\end{equation}
where we identify Radon measures on $D$ and $\Omega$ via  $\phi$.

\subsection{The proof of Proposition \ref{prop: main}}
Due to   \eqref{eq: laplacian}, for any harmonic function $v:\R^2\to \R$, the composition $u=v\circ \phi ^{-1}$ satisfies $\mLaplace _{\Omega}(u) =0$, thus
$u$ is a \emph{harmonic function on $\Omega$}.

Due to  the regularity of harmonic functions on RCD spaces, the function $u:\Omega\to \R$ is locally Lipschitz
(\cite[Theorem 1.1]{Ji14}, which is  a combination
of \cite[Theorem 6.2]{AGS}, \cite[Corollary 2.3]{AGS2} and \cite[Theorem 1.1, Proposition 5.1]{KosShan}). Applying this observation to the coordinate functions $v_{1,2}$ we deduce that $\phi^{-1} :\Omega \to D$ is locally Lipschitz.  Since
$\frac 1 f  \circ \phi^{-1}$ is a weak minimal upper gradient of $\phi ^{-1} :\Omega \to D$,
we deduce that $f^{-1}$ is locally bounded on $D$.

Let now $v:D\to \R$ be a harmonic function. Consider the composition   $u=v\circ \phi ^{-1}  \in N^{1,2} (\Omega)$. By \eqref{eq: laplacian}, $\mLaplace _{\Omega}(u) =0$.
Due to \eqref{eq: gradients} and \eqref{eq: equalitym} the right hand side of \eqref{eq: proposition} is given by $2\kappa \cdot \rho _u ^2 \cdot \mathcal H^2_{\Omega}$,
where $\rho _u$ is the weak minimal upper gradient of $u$.  Using  \eqref{eq: gradients}  again, it remains to show the following claim for
 any open subset $O\subset \bar O \subset \Omega$  .
A representative of $\rho _u ^2$ is contained in  the domain
of the Laplacian in
$N^{1,2}  (O)$ and
  we have the comparison of measures
	\begin{equation} \label{eq: O}
\mLaplace _{O} (\rho _u ^2) \geq 2 \kappa \cdot \rho _u ^2 \cdot \mathcal H^2 _{O} \,.
\end{equation}
The proof of \eqref{eq: O} follows from the Bochner inequality \cite[Proposition 4.36]{Gigli-memoirs}, \cite[Remark 6.3]{AGS}, \cite{AGMR} by localization, as follows.

 If the function $u$ is a restriction to $O$ of a \emph{test function}  $\hat u \in N^{1,2} (X)$  in the sense of
\cite[Equation (7.2)]{Ambrosio-review} then \eqref{eq: O} is precisely  \cite[Proposition 4.36]{Gigli-memoirs}, since $\mLaplace _O(u)=0$.  In general, we multiply $u$ by  a   test function which is constant $1$ on $O$ and has support in $\Omega$  (the existence of such cut-off functions has been verified in  \cite[Lemma 6.7]{local}).  This provides a function $\hat u\in N^{1,2} (X)$ which
restricts to  $u$ on $O$  and is a test function on $X$,  \cite[Proposition 4.17, Theorem 4.29]{Gigli-memoirs}.

This finishes the proof of Proposition \ref{prop: main}.

\subsection{An analytic conclusion}
By a combination of a smoothing  argument and a pointwise computation we are going to conclude:

\begin{cor} \label{cor: conclusion}
The  function $\log (f^2)$ is contained in  $L^1 _{loc} (D)$.  The distributional Laplacian 
$\Delta (\log (f^2))$ is a Radon measure on $D$ and satisfies
\begin{equation} \label{eq: maineq}
\Delta (\log (f^{2}))  \leq  - {2\kappa} \cdot {f^{2}}\cdot \mathcal H^2 _{D} \,.
\end{equation}
 For any domain $O$, compactly contained in $D$ there exist a sequence
of smooth functions $f_n:O\to (0,\infty)$ such that $f_n$ satisfy \eqref{eq: maineq} on $O$,
such that $\log(f_n)$ converges to $\log(f)$ in $L^1 (O)$
and such that $\Delta (\log (f_n ^2))$ weakly  converge to $\Delta (\log (f^2))$ as measures on $O$.
\end{cor}

\begin{proof}
Consider the function $h= f^{-2}$.  Due to Proposition \ref{prop: main}, the function $h$ is locally bounded on $D$. Moreover, for any harmonic function $v$ on $\R^2$, the function $h$  satisfies in the  sense of measures on $D$
\begin{equation} \label{eq: first} 
	\Delta (h \cdot |\nabla v|^2) \geq 2  \kappa \cdot |\nabla v|^2\cdot \mathcal H^2 _{D} \, .
	\end{equation}

We fix a ball $O=B_{1-\delta} (0) \subset D$ for the 
rest of the proof.
For all $\delta >2\cdot \epsilon >0$ consider the  mollifications $h_{\epsilon} :O\to \R$ obtained by a convolution of $h$ with the usual smooth  mollifiers
$\rho _{\epsilon} :\R^2 \to [0,\infty)$ whose  support is  in $B_{\epsilon} (0)\subset \R^2$.
Since the function $h$ is positive almost everywhere, the smooth functions $h_{\epsilon}$ are positive on $O$.
By a direct computation (or using the observation that
\eqref{eq: first} is a system of linear inequalities on the function $h$, which is moreover, equivariant with respect to translations), we see
that $h_{\epsilon}$ satisfies \eqref{eq: first}  on $O$ for all  harmonic functions $v$ on $\R^2$.

The function $h$ is bounded from above on the $\epsilon$-neighborhood of $O$, hence so is $\log (h)$. On the  other hand,  $-\log (h) =\log (\frac 1 h) \leq \frac 1 h$.
Thus, the integrability of $f^2=\frac 1 h$ shows  $\log (h) \in L^1 (O)$.

The convexity of the function $t\to \frac 1 t$ and Jensen's inequality  imply 
$$\frac 1 {h_{\epsilon}} \leq \left (\frac  1 h \right ) _{\epsilon} \,,$$
where on the right side we have the mollifications of the function $\frac 1 h$.
Since $(\frac 1 h) _{\epsilon}$ converge in $L^1 (O)$ to $\frac 1 h$ as $\epsilon $ goes to $0$,
we deduce that $\frac 1 {h_{\epsilon} }$ converges to $\frac 1 h$ is $L^1 (O)$ as well.

Similarly, $t\to -\log (t)=\log (\frac 1 t )$ is convex and arguing as above with Jensen's inequality
we see that $\log (h_{\epsilon} )$ converges in $L^1(O)$ to $\log (h)$ as $\epsilon$ goes to $0$.

For all $n > \frac 2 {\delta}$ consider  the smooth function $f_n$ on $O$ such that $f_n^{-2} =h_{\frac 1 n}$.
Like above, $\log f_n =-2\log h_{\frac 1 n}$ converge to  $\log f$ in $L^1 (O)$.  The remaining statements follow by a continuity argument,
once we have verified  $\Delta (\log (f_n^2)) \leq -2 \kappa \cdot f_n^2$ on $O$. 

As we have seen above, the smooth positive functions $h_{\frac  1 n}$ satisfy \eqref{eq: first}  on $O$, for all harmonic functions $v :\R^2\to \R$.
Therefore, we only need to verify that for a \emph{smooth positive} function $h$ on a domain $U$ in $\R^2$ the validity of  \eqref{eq: first}  for all harmonic functions $v :\R^2\to \R$
implies 	
\begin{equation} \label{eq: smoothcase}
	\Delta (\log (h)) =\frac {\Delta (h)} {h} - \frac {|\nabla h |^2} {h^2} \geq \frac {2\kappa} {h} \;.
	\end{equation}
It is sufficient to verify this  pointwise statement at a single point  $z\in U$ which we may assume to be  $z=0$.
	
	Set $e:= \nabla h (0)$.  If $e=0$ we choose $v$ to be a linear non-zero function on $\R^2$.  Then \eqref{eq: first}  implies $\Delta (h)(0)  \geq 2\cdot \kappa$, hence,
	\eqref{eq: smoothcase}.

	If $e\neq 0$, 	fix $\lambda \in \R$  and consider the uniquely determined
	symmetric traceless matrix $A:\R^2\to \R^2$ with
	$A(e)=\lambda \cdot e$. The function
	$$v(z):= \left\langle z,A(z) +e\right\rangle$$
	is harmonic on $\R^2$ and satisfies $\nabla v(z)= e+ 2\cdot A(z)$.
	Therefore,
	$$\nabla ( |\nabla v|^2 ) (z)=  4\lambda \cdot e + 8 \cdot  A^2(z)\; \;  \text {and} \; \; \Delta (|\nabla v|^2) (z)= 16 \lambda ^2 \,.$$
	Thus, the right hand side of \eqref{eq: first}  at $z=0$ is just $2\kappa \cdot |e| ^2$.
	
	For the left hand side of \eqref{eq: first} at the point $z=0$ we compute
	$$\Delta (h) \cdot |\nabla v|^2 + 2 \left\langle\nabla h, \nabla
	(|\nabla v| ^2)\right\rangle + h  \cdot \Delta (|\nabla v|^2) = $$
	$$= \Delta (h) (0) \cdot |e|^2 + 8\lambda \cdot |e^2| + 16\cdot \lambda ^2 \cdot h(0) \,.$$
	For  $\lambda = -\frac { |e|^2 } {4\cdot h(0)} $ the inequality \eqref{eq: first} now reads
	$$|e|^2 \cdot \left(\Delta (h) (0) - \frac {|e|^2}  {h(0)}\right) \geq 2\kappa \cdot |e|^2\, .$$
	Dividing by $h(0) \cdot  |e|^2$ we deduce \eqref{eq: smoothcase}.
	
	This finishes the proof of Corollary \ref{cor: conclusion}.
\end{proof}


\section{Curvature bound in the regular part} \label{sec: 6}
\subsection{Preliminaries from Alexandrov geometry} We refer to   \cite{BBI01} for basics about Alexandrov geometry and just agree on  notation here, following \cite{Petrunincomplete}.
 Let $\kappa$ be a fixed real number as before.  For points $p,x_1,x_2$ in a metric space $Y$, we denote
by $\tilde {\angle }^{\kappa } (p ^{x_1} _{x_2})$  the \emph{$\kappa$-comparison angle}, whenever it exists. Thus,  $\tilde {\angle }^{\kappa } (p ^{x_1} _{x_2})$ is the angle in the constant curvature surface  $M^2_{\kappa}$ at the vertex
$\tilde p$    of a triangle $\tilde p\tilde x_1\tilde x_2$ with the same side-lengths  as $px_1x_2$.

  A subset $O$ of a metric space $Y$ satisfies the \emph{$(1+3)$-points comparison} if for any quadruple of points $p;x_1,x_2,x_3 \in O$
	 the inequality
	$$\tilde {\angle }^{\kappa } (p ^{x_1} _{x_2}) + \tilde {\angle }^{\kappa } (p ^{x_2} _{x_3}) + \tilde {\angle }^{\kappa } (p ^{x_3} _{x_1}) \leq 2\pi \,$$
	holds true or one of the $\kappa$-comparison angles is not defined.
	
A   metric space $Y$ has curvature $\geq \kappa$   if
every point $y\in Y$ has a neighborhood $O$ which satisfies the  \emph{$(1+3)$-points comparison} and such that every pair of points in $O$ is connected in $Y$ by a geodesic.

A  complete geodesic metric space of curvature $\geq \kappa$ is called an \emph{Alexandrov space} of curvature $\geq \kappa$.

 If $Y$ has curvature $\geq \kappa$, the ball $\bar B_{2R}(y) \subset Y$ is compact and any pair of points in $B_R(y)$  is connected in $Y$ by a geodesic, then  $B_R(y)$ satisfies the $(1+3)$-point comparison, \cite[p. 3]{Petrunincomplete}.

Note finally that the $(1+3)$-point comparison property  (of subsets) is stable under Gromov--Hausdorff convergence.

The aim of this section is the  following
 \begin{prop} \label{prop: regprop}
 	The  subspace $X^{\delta}$ of $\delta$-regular points in $X$
		has curvature bounded from below by $\kappa$.
\end{prop}

 \subsection{Reshetnyak's theory} We continue to use the notation from  Section \ref{sec: 4} and Section \ref{sec: 5}. For  a  homeomorphism 
$\phi:D\to \Omega \subset X $ the length  $\ell _X (\phi \circ \gamma )$ is given by \eqref{eq:length}, 
for almost all curves $\gamma $  in $D$. The conformal factor  $f=\rho _{\phi}$ satisfies  the conclusion of Corollary \ref{cor: conclusion}.

By Corollary \ref{cor: conclusion}, the function $\Delta (\log f)$  is a Radon measure. Thus, on any compactly contained domain $O\subset D$ we can 
canonically represent $\log (f)$ as a sum of a harmonic function and a \emph{Riesz potential},
  \cite[p. 99]{Reshetnyak-GeomIV}. Using this representative of $f$,  the $f$-length $\ell _f(\gamma)$ of \emph{any} rectifiable curve
 $\gamma$ on $O$   is defined  in \cite[p. 100]{Reshetnyak-GeomIV}  by the formula  \eqref{eq:length}.

This induces a new metric $d_f$ on $O$ by letting $d_f (z_1,z_2)$ be the
infimum of all $f$-lengths of rectifiable curves connecting $z_1$ and $z_2$.  The metric $d_f$ induces  the original Euclidean topology on $O$, \cite[Theorem 7.1.1]{Reshetnyak-GeomIV}.
We define the metric space  $O_f =(O,d_f)$  to be the disc $O$ equipped with the metric $d_f$.

The following statement is implicitly contained  in \cite[p. 140]{Reshetnyak-GeomIV}. For convenience of the reader, we reduce the result to other more explicit   statements  in  \cite{Reshetnyak-GeomIV}.
\begin{lem} \label{lem: equalm}
The space $O_f$ has curvature bounded from below by $\kappa$.
\end{lem}

\begin{proof}
We find 
smooth positive functions $f_n:O\to \R$ approximating $f$ as in 
  Corollary \ref{cor: conclusion}.   By   \cite[Theorem 7.3.1]{Reshetnyak-GeomIV},  the distance functions $d_{f_{n}}$ converge on $O$
  locally uniformly against $d_f$.
  
  On the  other hand, $d_{f_n}$ is a smooth Riemannian metric. Its curvature can be computed pointwise by the classical  formula $-\frac {\Delta (\log (f_n))} {f_n^2}$,
  \cite[p. 40]{Reshetnyak-GeomIV}. Thus,  \eqref{eq: maineq} shows that  the Riemannian manifold $O_{f_{n}}$ has sectional
  curvature bounded from below by $\kappa$.

 Hence, $O_{f_{n}}$    has curvature bounded from below by $\kappa$ in the sense of Alexandrov, \cite[Theorem 6.5.6]{BBI01}.   Therefore, for any compact metric ball $B$ in
$(O,d_f)$ and the concentric ball $B'$ with  one-third the radius of $B$,
the set $B' \subset O_{f_{n}}$ satisfies the $(1+3)$-point comparison,
for all large $n$. By continuity,  $B'$   satisfies the $(1+3)$-point comparison,  when considered as a subset of $O_f$.
This completes the proof.
\end{proof}

\subsection{Sobolev-to-Lipschitz}  By construction, for \emph{almost all} curves $\gamma$ in $O$,
the length of $\phi \circ \gamma $ in $X$ is equal to the $f$-length
of $\gamma$, hence to the length of $\gamma$ in the metric space $O_f$.
The easiest way
 to upgrade the equality statement from \emph{almost all} curves to \emph{all} curves,   is via an  application of the
\emph{Sobolev-to-Lipschitz property} of RCD-spaces, stated as follows, \cite[p. 48]{Gigli}, \cite{AGS2}:

For any $RCD(\kappa, 2)$ space $X$, any open subset $W$ of $X$ and any  $u\in  N^{1,2} (W)$ for which the constant function $\rho=1$ is a weak upper gradient, the function $u$ has a  locally $1$-Lipschitz representative.

In fact, the Sobolev-to-Lipschitz property is defined and verified in \cite[p. 48]{Gigli} only in the global case $W=X$, but the proof presented there covers the local version formulated above.

\begin{lem} \label{lem: 1lip}
The open embedding	 $\phi :O_f \to  X$
is a local isometry.
\end{lem}

\begin{proof}
Set $W= \phi (O_f)$ and 
 consider the inverse map $\psi =\phi ^{-1} :W \to O_f$.
	By construction,  $\psi$ preserves the length of almost any curve $\gamma$ in $W$.  Therefore, the map
	$\psi$ is contained in the Sobolev space $N^{1,2} (W, O_f)$ and the constant function $1$ is a weak upper gradient of $\psi$.
	
	Therefore, for any point $y\in O_f$, the composition $\psi _y \in N^{1,2} (W)$
	of $\psi$ with the distance function in $O_f$ to the point $y$ has the constant
	function $1$ a weak upper gradient. By the Sobolev-to-Lipschitz property, this implies that $\psi _y$ is locally $1$-Lipschitz.  Since $X$ is a geodesic space,
	and $y$ was arbitrary, this implies that $\psi$ is locally $1$-Lipschitz.
	
	In order to prove that $\phi$ is locally $1$-Lipschitz, we apply the same argument. (Alternatively, 	this can be seen directly, as in \cite[Lemma 9.3]{LW-isoparametric}).
	Firstly, $\mathcal H^2_ {O_f} = f^2\cdot \mathcal H^2_O$ and the same computation as in Section \ref{sec: 4} shows that
	any  family of curves of modulus $0$ in $O$  has modulus $0$ in $O_f$.
	Thus, the map $\phi$ preserves the lengths of almost all curves in $O_f$.  Therefore, $\phi \in N^{1,2} (O_f, X)$ and the constant function $1$ is a  weak upper gradient of $\phi$.
	 Arguing as above we deduce that
	$\phi$ is locally $1$-Lipschitz, once the Sobolev-to-Lipschitz property has been verified locally in $O_f$.

But    any point in $O_f$ has a compact  neighborhood isometric to an Alexandrov space, by  \cite[Theorem 7.1.3]{Petrunin-semi}. This implies the Sobolev-to-Lipschitz property, as a consequence of \cite{Petrunin} and \cite[p. 40]{Gigli}.
	

Since $\phi$ and $\phi ^{-1}$ are locally $1$-Lipschitz, $\phi$ is  a local isometry.
\end{proof}

For any $\delta$-regular point $x_0\in  X$, we choose a domain  $\Omega$ containing $x_0$ as in Section \ref{sec: 4}.
Lemma  \ref{lem: 1lip} and Lemma \ref{lem: equalm} imply the existence of a neighborhood of  $x_0$ in $\Omega$ which has curvature bounded below by $\kappa$.
This  finishes  the proof of Proposition \ref{prop: regprop}.

\section{Extension to the singular points}
\subsection{Topological statement} Most of this rather long section is devoted to the proof of the following topological statement.

 \begin{prop} \label{prop: final}
	Any point $z \in \partial X$ with $b(z)=\frac 1 2$ has an open neighborhood $U$ in $X$  homeomorphic to
	the closed  half-plane $H$, such that $z$ lies in  the boundary line $\partial H$.
\end{prop}

Before we embark on the proof of this proposition, we explain why this
statement is sufficient to finish the proof of our main theorem.

\begin{lem}
 The validity of Proposition \ref{prop: final} implies that the set
 $X^{\delta}$ is strongly convex in $X$. Thus any geodesic $\gamma $ with endpoints  $x,y \in X^{\delta}$ is completely contained in $X^{\delta}$.	
\end{lem}

 \begin{proof}
 	Assume the contrary and consider  a point $z\in  X\setminus X^{\delta}$ on $\gamma $ closest to the point $x$.

 	Due to Lemma \ref{lem: nogeod}, the point $z$ is  contained in $\partial X$ and satisfies $b(z)=\frac 1 2$. Applying Proposition
 	\ref{prop: final} we find a neighborhood $U$ of $z$ homeomorphic   to  $H$, such that $z$ lies on the boundary line  $\partial H$.
 	
 	The part $\gamma ^+$ of the geodesic $\gamma $ between $x$ and $z$ is contained in $X^{\delta}$ (up to the point $z$) hence it intersects $\partial H$ only in the point $z$.
 	By choosing the neighborhood $U$ smaller we can therefore assume that $\gamma ^+$ separates $U$ into two components.
 	
 	Fix a point $q$ on the part of $\gamma$ between $z$ and $y$ sufficiently close to $z$ and let $U^+$ be the component of $U\setminus \gamma ^+$ which does not contain $q$. Thus, for any point $m \in U^+$, sufficiently close to $z$,
any geodesic  	between $q$ and $m$  intersects $\gamma ^+$.
 	
 	 	In particular, the point  $z$ lies on a geodesic between $m$ and $q$.  This statement, valid for the fixed  point  $q$ and an open set of points $m$,
 	contradicts the essentially non-branching property of the RCD  space $X$, \cite[Theorem 1.2, Corollary 1.4]{GRS}.
 \end{proof}
 	
 Therefore, the validity of Proposition \ref{prop: final} implies that the space $X^{\delta}$ is a geodesic space. Since $X$ is the completion of $X^{\delta}$, we would deduce from \cite{Petrunincomplete} that $X$ is an Alexandrov space with curvature $\geq \kappa$ and finish the proof of Theorem \ref{thm: main}.

\subsection{Interior points}
Consider a point  $x \in  X \setminus \partial X$. Assume that $x$ is not contained in $X^{\delta}$.
By Lemma \ref{lem: delta}, we find some $r>0$ such that $B_{2r} (x) \setminus X^{\delta}$ contains only the point $x$.  By Proposition \ref{prop: regprop}, $B_{2r} (x) \setminus \{x \}$ has curvature $\geq \kappa$ in the sense of Alexandrov.

 Due to Lemma \ref{lem: nogeod},
 for any $y,z\in B_r (x) \setminus \{x\}$, any geodesic connecting $y$ and $z$ in $X$ is contained in  $B_{2r} (x) \setminus \{x \}$.
 Toponogov's globalization theorem in the version of \cite{Petrunincomplete}, now shows that  $B_{r} (x) \setminus \{x \}$
satisfies the $(1+3)$-point comparison.   By continuity, $B_r(x)$ satisfies the $(1+3)$-point comparison as well.

We have just verified:
\begin{cor} \label{cor: inner}
The subspace $X\setminus \partial X$ of $X$ has curvature $\geq \kappa$.
\end{cor}

\subsection{Setting} We now fix a point $z\in \partial X$ with $b(z)=\frac 1 2$.  We consider a sequence of points $x_i \in X$ converging to $z$ and
a sequence $r_i$ of positive numbers converging to $0$. After choosing a subsequence we may and will assume
that the blow-up $(Y,y)=\lim (\frac 1 {r_i} X,x_i)$ exists.

  Furthermore, we consider the sequence of non-negative numbers $s_i=d(x_i,\partial X)$. By choosing a further subsequence, we may and will assume that the  following limit exists:
  $$A:=\lim _{i\to \infty} \frac {s_i} {r_i} \in [0,\infty]\,.$$
As we have seen in Section \ref{sec: 2}, any ball of any radius $t$ in the blow-up $Y$ has $\mathcal H^2$-measure at least $\frac {\pi}  2 t^2$.  From  the volume-cone rigidity, \cite{Volumerig}, and Lemma  \ref{lem: boundclosed} we deduce:
\begin{lem} \label{lem: boundexist}
	 The $RCD(0,2)$ space $Y$ can have non-empty boundary only if $Y$ is isometric to the flat half-plane $H$.
\end{lem}

Our next aim is to show that (non) boundary points of $X$ converge to (non) boundary points  in the blow-up $Y$. After that we will show that the blow-up $Y$ is isometric to a plane or a half-plane.

\subsection{Stability of the boundary}
In the previous notation we are going to show:
\begin{lem} \label{cor: regtobound}
If $A >0$ then the point $y$ is not  contained in $\partial Y$.
\end{lem}

\begin{proof}
	Assume the contrary. By Lemma \ref{lem: boundexist}, $Y$ is isometric to the half-plane $H$ and $y\in \partial H$.
	
	By  rescaling, we may assume that $A>9$. Thus, for all $i$ large enough, the compact ball $\hat B_i$ in $X_i=\frac 1 {r_i} X$ of radius $6$ around $x_i$ does not contain boundary points.  Due to Corollary \ref{cor: inner}, the open ball $B_i$ of radius $3$ around $x_i$ in $X_i$ satisfies the $(1+3)$-point comparison, see Section \ref{sec: 6} and \cite{Petrunincomplete}.

 A contradiction to the fact that the surfaces without boundary $B_i$ converge to a surface with boundary
(the ball in $H$ around the limit point $y$) can now be deduced in several ways.
We choose a way relying on   (the simplest case of) Perelman's topological stability theorem, \cite{P2}, \cite{Kap}.

From   \cite[Theorem 7.1.3]{Petrunin-semi} we deduce the existence of some $\delta >0$  and closed compact convex subsets $C_i\subset B_i$ containing
the ball $B_{\delta } (x_i) \subset B_i$.  The spaces $C_i$ are compact $2$-dimensional Alexandrov spaces converging (after choosing a subsequence) to an Alexandrov space $C\subset H$ which  contains the ball $B_{\delta } (y) \subset H$.
By Perelman's stability theorem  \cite{P2}, \cite{Kap},  for all large $i$, there exists a homeomorphism $\Phi_i:C_i\to C$ close to the identity.  Since $B_{\delta} (x_i)$ is a $2$-manifold without boundary, we deduce that $y$ must have in $Y$ a neighborhood  homeomorphic to a $2$-manifold without boundary which is impossible.

This contradiction finishes the proof.
\end{proof}


	
	\begin{lem} \label{lem: boundtobound}
		If $A=0$ then $Y$ is isometric to the  half-plane $H$ and $y$ is on the boundary $\partial H$.
	\end{lem}
	
	\begin{proof}
  Consider points $z_i \in \partial X$ with $d(z_i,x_i)=s_i =d(\partial X,x_i)$.  The assumption $A=0$ implies that the points
  $z_i \in X_i =\frac 1 {r_i} X$ converge to the same point $y \in Y$.

  Since $z_i \in \partial X$, we see that the density $b(y)$ at $Y$ is at most $\frac 1 2$.  The volume-cone rigidity argument implies that
  $b(y)=\frac 1 2$ and $Y$ is isometric to  the Euclidean cone $T_yY$.
  It remains to show that $T_yY$  cannot be the Euclidean cone over the circle of length $\pi$.

		We assume that $Y$ is a Euclidean cone over a circle and are going to derive a contradiction.
		
		For all $i$ large  enough, there exist no points $p_i$ in $\partial X$  with
		$r_i<d(p_i,z_i) <2r_i$, since $Y\setminus \{y\}$ is locally Euclidean. Denote by $K_i$ the set $\partial X\cap \bar B_{r_i } (z_i)$ and by $\hat K_i$ the complement $\partial X\setminus K_i$.  
		Then, for all $i$ large enough, $K_i$ is compact and $\hat K_i$
		is closed in $\partial X$. Moreover,
		$$d(z,z_i)\geq t_i:=d(K_i,\hat K_i)\geq r_i \,.$$
		Consider points $k_i\in K_i$ and $\hat k_i \in \hat K_i$ realizing the distance between $K_i$ and $\hat K_i$
		and take the blow-up (choosing a subsequence)
		$(\hat Y , k)=\lim (\frac 1 {t_i} X,k_i)$.
		
		As before, the volume-cone rigidity implies that $\hat Y$  is  either a half-plane or a cone over a circle. However, the points $\hat k_i$ converge (after taking a subsequence) to a non-Euclidean point $\hat k\in \hat Y$ with distance $1$ to $k$.  Therefore,  $\hat Y$ must be isometric to $H$ and  the
		geodesic between $k$ and $\hat k$ must lie on the boundary $\partial H$.
		
		Hence, the midpoints $m_i$ of any geodesic between $k_i$ and $\hat k_i$ in $X$ converge to a point on the boundary of $\hat Y$.
		
		But, by construction, the point $m_i$ in $X$ has distance $\frac {t_i} 2$ from $\partial X$.  We obtain a contradiction
		with Lemma  \ref{cor: regtobound} and finish the proof.
		\end{proof}

\begin{lem} \label{lem: bounddisappear}
If $A=\infty$ then $Y$ is isometric to $\R^2$.
\end{lem}

\begin{proof}
	Due to Lemma \ref{lem: boundtobound} and Lemma \ref{cor: regtobound},
	the sequence $(\frac 1 {s_i} X,x_i)$ converges to a half-plane $(H,\hat y)$ where $\hat y$ has distance $1$ from $\partial H$.
	Therefore, $\mathcal H^2 (B_1(\hat y)) =\pi$.
 Stability of the Hausdorff measures implies, $$\frac {\mathcal H^2 (B_{s_i} (x_i))} {s_i^2} \to \pi \,.$$
	Applying the Bishop--Gromov  inequality, and the assumption $A=\infty$, we see that for any fixed $t>0$,
	$$\frac { b(x_i, t \cdot r_i)} {r_i}  \to  \pi t^2 \,.$$
	Thus, the ball $B_t(y)$ in $Y$ has $\mathcal H^2$-measure $\pi\cdot t^2$. Since $t$ is arbitrary,
	 the volume-cone rigidity  implies that $Y$ is isometric to $T_yY =\R^2$.
\end{proof}

 Combining the last three statements we easily arrive at
 \begin{cor} \label{cor: combination}
  The space $Y$ is either $\R^2$ or the half-plane $H$. The case $Y=H$ happens if and only if $A<\infty$. Moreover, in this case $\partial H$  coincides with the limit  of the boundary $\partial X$.
 \end{cor}

\begin{proof}
	The only statement not directly contained in Lemma \ref{cor: regtobound}, Lemma \ref{lem: boundtobound} and Lemma \ref{lem: bounddisappear} is that for $0<A<\infty$ the space $Y$ is isometric to $H$.  However, if $A<\infty$ we can replace the base points $x_i$ by closest points $z_i \in \partial X$ and apply  Lemma \ref{lem: boundtobound} to deduce the statement.
\end{proof}

 \subsection{Reifenberg's lemma twice}
 We can now finish the

\begin{proof}[Proof of Proposition \ref{prop: final}]
  Due to Corollary \ref{cor: combination}, for any sequence $z_i \in \partial X$ converging to $z$ and any sequence $r_i \to 0$  the
  sequence $(\frac 1 {r_i}  \partial X, z_i)$ converges to the line $\partial H =\R$.
    Applying the  Cheeger--Colding--Reifenberg Lemma, \cite[Theorem A1]{CC}, we deduce that
  a neighborhood of $z$ in $\partial X$ is homeomorphic to an interval.

  Consider now the doubling $W= X\cup _{\partial X} X$ of $X$ with its natural length metric, compare \cite[Section 5]{P2}.

  We claim that for any sequence $x_i \in W$ converging to $z$ and any sequence of positive  numbers $r_i$ converging to $0$
the blow-up $\lim (\frac 1 {r_i} W,x_i)$ is isometric to $\R^2$.

  Choosing a subsequence and using the symmetry we may assume that $x_i \in X \subset W$ and that for $s_i:= d(x_i, \partial X)$ the quotients
  $s_i/r_i$ converge to a number $A\in [0,\infty]$.  Applying Corollary
  \ref{cor: combination}, we see that  in the case $A=\infty$, the blow-up coincides with $\lim (\frac 1 {r_i} X,x_i) =\R^2$.

  On the other hand, if $A<\infty$, we can change the base points $x_i$ and assume $x_i \in \partial X$, without changing the isometry class of the blow-up. Then the limit $\lim (\frac 1 {r_i} W,x_i)$ is the doubling of $H=\lim (\frac 1 {r_i} W,x_i)$ along  the boundary $\partial H=\lim
  (\frac 1 {r_i} \partial X, z_i)$.


Having proved the claim, we can now apply the Cheeger-Colding-Reifenberg Lemma a second time  and deduce
that a neighborhood $V$ of $z$ in $W$ is homeomorphic to an open disc.
Passing to a smaller subdisc if necessary, we obtain a homeomorphism between $V$ and the plane which takes $V\cap \partial X$ to a line.

By connectedness reasons we see that $X\cap V$ must be homeomorphic to a half-plane.
\end{proof}

\bibliographystyle{alpha}
\bibliography{Ricci+}




\end{document}